\documentclass[12pt]{amsart}
\usepackage{amsfonts}
\usepackage{amsmath}
\usepackage{amssymb}
\usepackage[margin=1.2in]{geometry}
\newcommand{\R}{\mathbb{R}}

\newcommand{\D}{\mathcal{D}'}
\newcommand{\Hk}{\mathcal{H}_{k}}

\newtheorem{theorem}{Theorem}[section]

\newtheorem{lemma}[theorem]{Lemma}
\newtheorem{proposition}[theorem]{Proposition}

\theoremstyle{definition}

\newtheorem{remark}[theorem]{Remark}

\numberwithin{equation}{section}

\begin{document}
\title[Examples in the theory of Beurling's generalized primes]{Some examples in the theory of Beurling's generalized prime numbers}

\author[G. Debruyne]{Gregory Debruyne}
\thanks{G. Debruyne gratefully acknowledges support by Ghent University, through a BOF Ph.D. grant}
\address{G. Debruyne\\ Department of Mathematics\\ Ghent University\\ Krijgslaan 281 Gebouw S22\\ B 9000 Gent\\ Belgium}
\email{gdbruyne@cage.UGent.be}
\author[J.-C. Schlage-Puchta]{Jan-Christoph Schlage-Puchta}
\address{J.-C. Schlage-Puchta\\ Institut f\"{u}r Mathematik\\ Universit\"{a}t Rostock\\ 18051 Rostock\\ Germany}
\email{jan-christoph.schlage-puchta@uni-rostock.de}
\author[J. Vindas]{Jasson Vindas} 
\thanks{The work of J. Vindas was supported by the Research Foundation--Flanders, through the FWO-grant number 1520515N}
\address{J. Vindas\\ Department of Mathematics\\ Ghent University\\ Krijgslaan 281 Gebouw S22\\ B 9000 Gent\\ Belgium}
\email{jvindas@cage.UGent.be}

\begin{abstract}
Several examples of generalized number systems are constructed to compare various conditions occurring in the literature for the prime number theorem in the context of Beurling generalized primes.
\end{abstract}

\maketitle

\section{Introduction}

In this article we shall construct various examples of generalized number systems in order to compare three major conditions for  the validity of the prime number theorem (PNT) in the setting of  Beurling's theory of generalized primes.

Beurling's abstract formulation of the PNT is as follows \cite{bateman-diamond,beurling}. A set of \emph{generalized primes} is simply a sequence $P=\left\{p_k\right\}_{k=1}^{\infty}$ of real numbers tending to infinity with the only requirement that $1<p_{1}\leq p_{2}\leq \dots$. Its associated set of \emph{generalized integers} is the non-decreasing sequence $1=n_{0}<n_{1}\leq n_{2}\leq \dots$ arising as  all possible finite products of the generalized primes (occurring in $\{n_{k}\}_{k=1}^{\infty}$ as many times as they can be represented by  $p_{\nu_{1}}^{\alpha_1}p_{\nu_{2}}^{\alpha_2}\dots p_{\nu_m}^{\alpha_m}$ with $\nu_{j}< \nu_{j+1}$). Consider the counting functions of the generalized integers and primes,
\begin{equation}
\label{ibpneq1}
N(x)=N_{P}(x)=\sum_{n_{k}\leq x}1 \ \ \ \mbox{ and } \ \ \ \pi(x)=\pi_{P}(x)=\sum_{p_{k}\leq x}1\: ,
\end{equation}
where one takes multiplicities into account.
Beurling's problem is then to determine asymptotic requirements on $N$, as minimal as possible, which ensure the PNT in the form
\begin{equation}
\label{ibpneq3}
\pi(x)\sim\frac{x}{\log x}\: , \ \ \ x\rightarrow\infty\: .
\end{equation}

Three chief conditions on $N$ are the following ones. The first of such was found by Beurling in his seminal work \cite{beurling}. He showed that 
\begin{equation}
\label{ibpneq4}
N(x)=ax+O\left(\frac{x}{\log^{\gamma}x}\right)\: , \ \ \ x\to\infty\: ,
\end{equation}
where $a>0$ and $\gamma>3/2$, suffices for the PNT (\ref{ibpneq3}) to hold. A significant extension to this result was achieved by Kahane \cite{kahane1}. He proved, giving so a positive answer to a long-standing conjecture by Bateman and Diamond \cite{bateman-diamond}, that 
the $L^{2}$-hypothesis
\begin{equation}
\label{kahaneeq}
\int_{1}^{\infty}\left|\frac{\left(N(t)-at\right)\log t}{t}\right|^{2}\frac{\mathrm{d}t}{t}<\infty\: ,
\end{equation}
for some $a>0$, implies the PNT. We refer to the recent article \cite{zhang2014} by Zhang for a detailed account on Kahane's proof of the Bateman-Diamond conjecture (see also the expository article \cite{diamond2011}). Another condition yet for the PNT has been recently provided by Schlage-Puchta and Vindas \cite{s-v}, who have shown that 
\begin{equation}
\label{ibpneq5}
N(x)=ax+O\left(\frac{x}{\log^{\gamma}x}\right)\ \ \ \ (\mathrm{C})\: , \ \ \ x\to\infty\: ,
\end{equation}
with $a>0$ and $\gamma>3/2$ is also sufficient to ensure the PNT. The symbol $(\mathrm{C})$ stands for the \emph{Ces\`{a}ro sense} \cite{estrada-kanwal} and explicitly means that there is some (possibly large) $m\in\mathbb{N}$ such that the following average estimate holds: 
\begin{equation}
\label{ibpneq6}
\int_{1}^{x} \frac{N(t)-at}{t}\left(1-\frac{t}{x}\right)^{m}\mathrm{d}t=O\left(\frac{x}{\log^{\gamma}x}\right)\: , \ \ \ x\to\infty\: .
\end{equation}
 
It is obvious that Beurling's condition (\ref{ibpneq4}) is a particular instance of both (\ref{kahaneeq}) and (\ref{ibpneq5}). Furthermore, Kahane's PNT also covers an earlier extension of Beurling's PNT by Diamond \cite{diamond1969}. However, as pointed out in \cite{s-v,zhang2014}, the relation between (\ref{kahaneeq}) and (\ref{ibpneq5}) is less clear. Our main goal in this paper is to compare (\ref{kahaneeq}) and (\ref{ibpneq5}). We shall construct a family of sets of generalized primes fulfilling the properties stated in the following theorem:

\begin{theorem} \label{thmain} Let $1<\alpha<3/2$. There exists a generalized prime number system $P_{\alpha}$ 
whose generalized integer counting function $N_{P_{\alpha}}$ satisfies (for some $a_{\alpha}>0$)
\begin{equation}
\label{Cesaroeqgi}
N_{P_\alpha}(x)=a_{\alpha}x+O\left(\frac{x}{\log^{n}x}\right) \ \ \ (\mathrm{C})\:, \ \ \ \mbox{for } n=1,2,3,\dots, 
\end{equation}
but violates $(\ref{kahaneeq})$, namely,
\begin{equation}
\label{nkahane}
\int_{1}^{\infty}\left|\frac{\left(N_{P_{\alpha}}(t)-a_{\alpha}t\right)\log t}{t}\right|^{2}\frac{\mathrm{d}t}{t}=\infty\: . 
\end{equation}
Moreover, these generalized primes satisfy the PNT with remainder
\begin{equation}
\label{ibpneq7}
\pi_{P_{\alpha}}(x) = \frac{x}{\log x} + O\left(\frac{x}{\log^{\alpha}x}\right)\ .
\end{equation}
\end{theorem}

Our method for establishing Theorem \ref{thmain} is first to construct examples of \emph{continuous} generalized number systems witnessing the desired properties. For it, we shall translate in Section \ref{aux} the conditions (\ref{Cesaroeqgi}) and (\ref{nkahane}) into analytic properties of zeta functions.  Our continuous examples are actually inspired by the one Beurling gave in \cite{beurling} to show that his theorem is sharp, that is, an example that satisfies (\ref{ibpneq4}) for $\gamma =3/2$ but for which the PNT  (\ref{ibpneq3}) fails. Concretely, in Section \ref{the continuous example} we study the associated zeta functions $\zeta_{C,\alpha}$ to the family of absolutely continuous Riemann prime counting functions
\begin{equation}
\label{Riemannconteq}
 \Pi_{C,\alpha}(x) = \int^{x}_{1} \frac{1-\cos(\log^{\alpha} u)}{\log u } \mathrm{d}u\:, \  \ \ x\geq1\: .
\end{equation}
If $\alpha = 1$ in (\ref{Riemannconteq}), this reduces to the example of Beurling, whose associated zeta function is $\zeta_{C,1}(s)=(1+1/(s-1)^{2})^{1/2}$. In case $\alpha>1$, explicit formulas for the zeta function of (\ref{Riemannconteq}) are no longer available, which makes its analysis considerable more involved than that of Beurling's example. In the absence of an explicit formula, our method rather relies on studying qualitative properties of the zeta function, which will be obtained in Theorem \ref{thlacseries} via the Fourier analysis of certain related singular oscillatory integrals. As we show, the condition $1<\alpha<3/2$ from Theorem \ref{thmain} is connected with the asymptotic behavior of the derivative of $\zeta_{C,\alpha}(s)$ on $\Re e\:s=1$.

The next step in our construction for the proof of Theorem \ref{thmain} is to select a discrete set of generalized primes $P_{\alpha}$ whose prime counting function $\pi_{P_{\alpha}}$ is sufficiently close to (\ref{Riemannconteq}). We follow here a discretization idea of Diamond, which he applied in \cite{diamond2} for producing a discrete example showing the sharpness of Beurling's theorem. We prove in Section \ref{the discrete example} that the set of generalized primes 
\begin{equation}
\label{eqdiscrete} P_{\alpha}=\{p_{k}\}_{k=1}^{\infty}\:, \ \ \ p_{k}=\Pi_{C,\alpha}^{-1}(k)\:,
\end{equation} 
satisfies all requirements from Theorem \ref{thmain}. 

Note that Diamond's example from \cite{diamond2} is precisely the case $\alpha=1$ of (\ref{eqdiscrete}). However, it should be also noticed that the analysis of our example (\ref{eqdiscrete}) that we carry out in Section \ref{the discrete example} is completely different from that given in \cite{diamond2}. Our arguments rely on suitable bounds for the associated zeta functions and their derivatives. Moreover, our ideas lead to more accurate asymptotic information for the generalized integer counting function of Diamond's example. We give a proof of the following theorem in Section \ref{Diamond example}.  

\begin{theorem}
\label{thDiamond} Let $P_{1}$ be the set of generalized primes $(\ref{eqdiscrete})$ corresponding to $\alpha=1$. There are constants $c$, $\{d_{j}\}_{j=0}^{\infty}$, and $\{\theta_{j}\}_{j=0}^{\infty}$ such that $N_{P_{1}}$ has the following asymptotic expansion
\begin{align}
\label{diamondexasympexp}
N_{P_{1}}(x)&\sim cx+ \frac{x}{\log^{3/2}x}\sum^{\infty}_{j=0}d_{j}\frac{\cos(\log x+\theta_{j})}{\log^{j}x}\\
\nonumber
&= cx+d_{0}\frac{x\cos(\log x+ \theta_{0})}{\log^{3/2} x}+ O\left(\frac{x}{\log^{5/2}x}\right), \quad x\to\infty\: ,
\end{align}
with $c>0$ and $d_{0}\neq0$, while the PNT does not hold for $P_{1}$.
\end{theorem}

We mention that Theorem \ref{thDiamond} not only shows the sharpness of $\gamma>3/2$ in Beurling's condition (\ref{ibpneq4}) for the PNT, but also that of $\gamma>3/2$ in (\ref{ibpneq5}). In addition, (\ref{diamondexasympexp}) implies that all Riesz means of the relative error $(N_{P_{1}}(x)-cx)/x$ satisfy
$$
\int_{1}^{x} \frac{N_{P_{1}}(t)-ct}{t}\left(1-\frac{t}{x}\right)^{m}\mathrm{d}t=\Omega_{\pm}\left(\frac{x}{\log^{3/2}x}\right), \quad x\to\infty\: , \quad m=0,1,2,\dots\: . 
$$
Observe also that Theorem \ref{thmain} in particular shows that the PNT by Schlage-Puchta and Vindas is a proper generalization Beurling's result. They gave an example in \cite[Sect. 6]{s-v} to support this result, but their proof contains a few mistakes (there are gaps in the proof of \cite[Lemma 6]{s-v} and the proof of \cite[Eq. (6.4)]{s-v} turns out to be incorrect). The last section of this article will be devoted to correcting these mistakes, we prove there:
\begin{theorem}
\label{propExample}
There exists a set of generalized primes $P^{\ast}$ such that $N_{P^{\ast}}(x)= x+\Omega(x/\log^{4/3}x)$, but $N_{P^{\ast}}(x)=x+O(x/\log^{5/3-\varepsilon} x)$ in Ces\`{a}ro sense for arbitrary $\varepsilon > 0$. Furthermore, for this number system we have $\pi_{P^{\ast}}(x)=x/\log x+O(x/\log^{4/3-\varepsilon} x)$. 
\end{theorem}

\subsection{Notation}
\label{notation} We will often make use of standard Schwartz distribution calculus in our manipulations. See the textbooks \cite{hormander1990, vladimirov} for the theory of distributions and \cite{estrada-kanwal,p-s-v} for asymptotic analysis of generalized functions. The standard test function spaces are denoted as usual by $\mathcal{D}(\mathbb{R})$ and $\mathcal{S}(\mathbb{R})$, while $\mathcal{D}'(\mathbb{R})$ and $\mathcal{S}'(\mathbb{R})$ stand for their topological duals, the spaces of distributions and tempered distributions. We fix the constants in the Fourier transform as
$\hat{\phi}(t)=\int_{-\infty}^{\infty}e^{-itx}\phi(x)\:\mathrm{d}x.$ Naturally, the Fourier transform is well defined on $\mathcal{S}'(\mathbb{R})$ via duality. If $f\in\mathcal{S}'(\mathbb{R})$ has support in $[0,\infty)$, its Laplace transform is
$\mathcal{L}\left\{f;s\right\}=\left\langle f(u),e^{-su}\right\rangle,
$ $\Re e\:s>0$, and its Fourier transform $\hat{f}$ is the distributional boundary value of $\mathcal{L}\left\{f;s\right\}$ on $\Re e\:s=0$. We use the notation $H$ for the Heaviside function, it is simply the characteristic function of $(0,\infty)$.

\section{Auxiliary lemmas}
\label{aux}

We begin by defining some other helpful number-theoretic functions. As usual, the zeta function is indispensable for studying the prime number theorem in this context,
\begin{equation}
\label{defzeta}
\zeta(s) =\int^{\infty}_{1^{-}} x^{-s}\mathrm{d}N(x)\: .
\end{equation}
Besides the usual prime counting function $\pi$, we will also work with the Riemann prime counting function,
\begin{equation}
\label{defriemann}
\Pi(x) = \sum^{\infty}_{n=1} \frac{\pi(x^{1/n})}{n}\: ,
\end{equation}
and we have the following link between $\Pi$ and $\zeta$,
\begin{equation}
\label{eqlinkzetariemann}
 \zeta(s) = \exp\left(\int^{\infty}_{1^{-}}x^{-s}\mathrm{d}\Pi(x)\right)\:.
\end{equation}

 We will consider an even broader definition of generalized primes \cite{beurling}, which also takes into account `continuous' number systems. So, in this sense a generalized prime number system is merely a non-decreasing function $\Pi$ that vanishes for $x\leq1$, where we assume that the integral involved in (\ref{eqlinkzetariemann}) is absolutely convergent in the half-plane $\Re e\: s>1$. We normalize $\Pi$ in such a way that it is right continuous. Clearly, the zeta function $\zeta$ from (\ref{eqlinkzetariemann}) can always be represented as (\ref{defzeta}) with a unique non-decreasing function $N$ if we impose that $N$ is right continuous; in fact, $N$ is determined by $\mathrm{d}N=\exp^{\ast}(\mathrm{d}\Pi)$, where the exponential is taken with respect to the multiplicative convolution of measures \cite{diamond1}. The function $\pi$ does not need to make sense in this framework. Note also that if $N$ satisfies any of the three conditions (\ref{ibpneq4})--(\ref{ibpneq5}), then $N(x)\sim
  ax$; consequently, if two of such conditions are simultaneously satisfied, the constant $a$ should be the same. We remark as well that all the three PNT discussed in the introduction are valid in this more general setting. 

In the rest of this section we connect (\ref{kahaneeq}) and (\ref{ibpneq5}) with the boundary behavior of $\zeta(s)$ on the line $\Re e\:s=1$. 
\subsection{A sufficient condition for the Ces\`{a}ro behavior}

The following Tauberian lemma gives sufficient conditions on the zeta function for $N$ to have the Ces\`{a}ro behavior (\ref{ibpneq5}) with $\gamma=n\in\mathbb{N}$. The proof of this result makes use of the notion of the quasiasymptotic behavior of Schwartz distributions; for it, we use the notation exactly as in \cite[Sect. 2.12, p. 160]{p-s-v} (see also \cite[p. 304]{s-v}). 

\begin{lemma} \label{lemsufcondcesaro} Let $n\in\mathbb{N}$. Suppose that the function $F(s) = \zeta(s) - a/(s-1)$ can be extended to the closed half-plane $\Re e\:s\geq 1$ as an $n$ times continuously differentiable function. If 
for every $ 0 \leq j \leq n$ the functions $F^{(j)}(1+it)$ have at most polynomial growth with respect to the variable $t$, then $N$ satisfies the Ces\`{a}ro estimate
\begin{equation}
\label{eqcesaron}
 N(x) = ax + O\left(\frac{x}{\log^{n}x}\right)\quad \mathrm{(C)}\:, \quad x \rightarrow \infty\: .
\end{equation}
\end{lemma}
\begin{proof} We define the function $R$ with support in $[0,\infty)$ in such a way that the relation $N(x) = axH(x-1) + xR(\log x)$ holds ($H$ is the Heaviside function). By the Wiener-Ikehara theorem (cf. \cite{korevaar,zhang2014WI}), the assumptions imply $N(x) \sim ax$. This ensures that $R\in\mathcal{S}'(\mathbb{R})$. A quick computation shows that $F(s) = a + s\mathcal{L}\{R;s-1\}$ for \mbox{$\Re $e $ s > 1$}. Let $\phi$ be an arbitrary test function from $\mathcal{S}(\R)$. We obtain
\begin{align*}
 \left\langle R(u+h), \phi(u)\right\rangle & = \frac{1}{2 \pi} \left\langle \hat{R}(t), \hat{\phi}(-t)e^{iht}\right\rangle\\
 & = \frac{1}{2 \pi} \lim_{\sigma \rightarrow 1^{+}} \int^{\infty}_{-\infty} \frac{F(\sigma + it) - a}{\sigma + it} \hat{\phi}(-t) e^{iht}\mathrm{d}t
\\
&
= \frac{1}{2 \pi}\int^{\infty}_{-\infty} \frac{F(1 + it) - a}{1 + it} \hat{\phi}(-t) e^{iht}\mathrm{d}t
\: .
\end{align*} 
 By using integration by parts $n$ times, we can bound this last term as
 \begin{equation*}
\frac{(-1)^{n}}{2\pi} \int^{\infty}_{-\infty} \left(\frac{F(1 + it)-a}{1+ it} \hat{\phi}(-t)\right)^{(n)}\frac{e^{iht}}{i^{n}h^{n}}\:\mathrm{d}t = O(h^{-n})\:,  \quad h\to\infty\:.
 \end{equation*}
 The last step is justified because all the derivatives of $F(1+it)$ have at most polynomial growth and any test function in $\mathcal{S}(\R)$ decreases faster than any inverse power of $|t|$. We thus find that $\int_{-\infty}^{\infty} R(u+h)\phi(u)\mathrm{d} u = O(h^{-n})$. Assuming that $\phi\in\mathcal{D}(\mathbb{R})$ and writing $h = \log \lambda$ and $\varphi(x) = e^{x}\phi(e^{x})$, we obtain the quasiasymptotic behavior
 \begin{equation} \label{eqquasi1} R(\log (\lambda x)) = O\left(\frac{1}{\log^{n}\lambda}\right), \quad \lambda \rightarrow \infty\:, \text{ in } \D(0,\infty)\: ,
\end{equation}
which explicitly means that 
\[ \int^{\infty}_{1} R(\log (\lambda x)) \varphi(x)\mathrm{d}x = O\left(\frac{1}{\log^{n}\lambda}\right), \quad \lambda \rightarrow \infty,
\]
for every test function $\varphi \in \mathcal{D}(0,\infty)$. Using \cite[Thm. 4.1]{vindas}, we obtain that the quasiasymptotic behavior (\ref{eqquasi1}) in the space $\D(0,\infty)$ is equivalent to the same quasiasymptotic behavior in the space $\D(\R)$, and, because of the structural theorem for quasiasymptotic boundedness \cite[Thm. 2.42, p. 163]{p-s-v} (see also \cite{vindas,vindas3}), we obtain the Ces\`{a}ro behavior (\ref{eqcesaron}).
\end{proof}

\subsection{Kahane's condition in terms of $\zeta$}

Note first that Kahane's condition (\ref{kahaneeq}) can be written as
\[ N(x) = ax + \frac{x}{\log x} E(\log x)\:, \quad  x \geq 1\: .
\]
where $E \in L^{2}(\R)$. We set $E(u)=0$ for $u<0$. Notice that $E(u)/u$ is continuous from the right at every point, as follows directly from its definition, and in particular it is integrable near $u=0$.

In the rest of this discussion we consider a generalized number system which satisfies Kahane's condition (\ref{kahaneeq}) with $a>0$. Since we have $N(x)\sim a x$, the abscissa of convergence of $\zeta$ is equal to $1$. Furthermore, $\zeta(1+it)$ always makes sense as a tempered distribution (the Fourier transform of the tempered measure $e^{-u}\mathrm{d}N(e^{u})$). With these ingredients we can compute the zeta function. We obtain 
\begin{equation} 
\label{eqformulazeta}
\zeta(s) =  \frac{a}{s-1} + a + sG(s) \quad \Re e \text{ } s > 1\: ,
\end{equation}
with 
\[ G(s) = \int^{\infty}_{0} e^{-(s-1)u} \frac{E(u)}{u} \mathrm{d}u\: .
\]
The function $G$ admits a continuous and bounded extension to \mbox{$\Re e$ $s = 1$}:
\[ G(1+it) = \int^{\infty}_{0} e^{-itu} \frac{E(u)}{u} \mathrm{d}u\: .
\]
Indeed since $E(u)u^{-1} \in L^{1}(\R)\cap L^{2}(\mathbb{R})$, its Fourier transform $G(1+it) \in C(\mathbb{R})\cap L^{\infty}(\R)\cap L^{2}(\mathbb{R})$. Furthermore,
\[ G'(1+it) = - \hat{E}(t)\in L^{2}(\R)\ .
\]

These observations lead to the following lemma. Recall that $H$ is the Heaviside function, so that $H(|t|-1)$ below is the characteristic function of $(-\infty,-1)\cup (1,\infty)$.
\begin{lemma} \label{lemnvwkahane} Kahane's condition $(\ref{kahaneeq})$ holds if and only if the boundary value distribution of $ \left(\zeta(s)-a/(s-1)\right)'$ on $\Re e\:s=1 $ satisfies 
\begin{equation} \label{eqnvwkahane1}
\left.\frac{d}{ds}\left(\zeta(s)-\frac{a}{s-1}\right)\right|_{s=1+it}\in L^{2}_{loc}(\R)
\end{equation}
and
\begin{equation}\label{eqnvwkahane2}
 \left(\frac{\zeta(1+it)}{t}\right)' H(|t|-1) \in L^{2}(\R)\: .
\end{equation}
\end{lemma}
Naturally, the derivative in (\ref{eqnvwkahane2}) is taken in the distributional sense with respect to the variable $t$. 
\begin{proof}
We have already seen that Kahane's condition holds if and only if  $G'(1+it)\in L^{2}(\mathbb{R})$, and that (\ref{eqnvwkahane1}) and (\ref{eqnvwkahane2}) are necessary for it. Assume this two conditions. Note that (\ref{eqnvwkahane1}) is sufficient to conclude $G(1+it)\in C(\mathbb{R})$, while (\ref{eqnvwkahane2}) and (\ref{eqformulazeta}) imply the bound 
$$G(1+it)=O(\sqrt{|t|}) \quad \mbox{for } |t|>1\: ,$$ 
because 
\[|\zeta(1+it)| \ll |t|\int_{1\leq|u|\leq |t|} \left|\left(\frac{\zeta(1+iu)}{u}\right)'\right|\mathrm{d}u\ll |t|^{3/2}, 
\]
by H\"{older}'s inequality. So, we may take the continuity of $G(1+it)$ and the bound $G(1+it)=O(\sqrt{|t|})$ for granted in the rest of the proof. In view of (\ref{eqformulazeta}), the function involved in (\ref{eqnvwkahane1}) is precisely $G(1+it)+(1+it)G'(1+it)$; therefore, (\ref{eqnvwkahane1}) yields $G'(1+it)\in L_{loc}^{2}(\mathbb{R})$. It remains to show that $G'(1+it)$ is square integrable on $\mathbb{R}\setminus[-1,1]$. For $|t|>1$, appealing again to the defining equation (\ref{eqformulazeta}), we have
\begin{align*}
 \frac{i(1+it)}{t}G'(1+it)&=\left(\frac{\zeta(1+it)}{t}\right)'
+\frac{2a}{it^{3}}+ \frac{(a+G(1+it))}{t^{2}}
\\
&
=\left(\frac{\zeta(1+it)}{t}\right)'+O\left(\frac{1}{|t|^{3/2}}\right)\in L^{2}\left(\mathbb{R}\setminus[-1,1]\right)\: ,
\end{align*}
which now gives $G'(1+it)\in L^{2}(\mathbb{R})$.  
\end{proof}

Our strategy in the next two sections to show Theorem \ref{thmain} is to exhibit examples of generalized number systems which break down the conditions from Lemma \ref{lemnvwkahane} but satisfy those from Lemma \ref{lemsufcondcesaro}.
\section{Continuous examples}
\label{the continuous example}
We shall now study the family of absolutely continuous Riemann prime counting functions (\ref{Riemannconteq}). For ease of writing, we drop $\alpha$ from the notation and we simply write
\begin{equation}
\label{Riemannconteq2} \Pi_{C}(x)=\Pi_{C,\alpha}(x)=\int^{x}_{1} \frac{1-\cos(\log^{\alpha} u)}{\log u } \mathrm{d}u\:, \  \ \ x\geq1\: .
\end{equation}
The number-theoretic functions associated with this example will also have the subscripts $C$, that is, we denote them as $N_{C}$ and $\zeta_{C}$. As pointed out in the Introduction, when $\alpha = 1$ we recover the example of Beurling. For this reason, it is clear that $\alpha = 1$ will not yield an example for Theorem \ref{thmain}, as the prime number theorem is not even fulfilled and hence neither holds the Ces\`{a}ro behavior (\ref{Cesaroeqgi}) for $N_{C}$ with $n>3/2$. We assume therefore in this section that $\alpha > 1$. Now we calculate the function $\zeta_{C}$ of our continuous number system via formula (\ref{eqlinkzetariemann}):
\begin{align*}
 \log \zeta_{C}(s) &= \int^{\infty}_{1} \frac{d\Pi_{C}(x)}{x^{s}} = \int^{\infty}_{1} \frac{1- \cos(\log^{\alpha}x)}{x^{s}\log x }\:\mathrm{d}x\\
 & = \int^{\infty}_{0} \frac{1- \cos u^{\alpha}}{u} e^{-(s-1)u}\mathrm{d}u\\
 & = \mathrm{F.p.} \int^{\infty}_{0} \frac{e^{-(s-1)u}}{u} \mathrm{d}u - \mathrm{F.p.} \int^{\infty}_{0} \frac{\cos u^{\alpha}}{u} e^{-(s-1)u}\mathrm{d}u\\
 & = -\log(s-1) - \gamma - K(s)\: , \quad \Re \text{e } s > 1\: ,
\end{align*}
where $\gamma=0.57721\dots$ is (\emph{from now on} in this article) the Euler-Mascheroni constant,
\begin{equation}\label{defK}K(s):=\mathrm{F.p.} \int^{\infty}_{0} \frac{\cos u^{\alpha}}{u} e^{-(s-1)u}\mathrm{d}u\: , \quad \Re \text{e } s >1\:,
\end{equation}
and $\mathrm{F.p.}$ stands for the Hadamard finite part of a divergent integral \cite[Sect. 2.4]{estrada-kanwal}.
Summarizing, we have found that
\begin{equation}\label{zetacontformula} \zeta_{C}(s) = \frac{e^{-\gamma}e^{-K(s)}}{s-1}\:, \quad \Re e \text{ } s > 1\:.
\end{equation}

It is clear that we must investigate the properties of the function $K$ in order to make further progress in understanding the zeta function $\zeta_{C}$ of (\ref{Riemannconteq2}). The next theorem is of independent interest, it tells us a number of useful analytic properties of the singular integral (\ref{defK}).
\begin{theorem} \label{thlacseries} Let $\alpha > 1$. The function $K$, defined by $(\ref{defK})$ for $\Re e\:s>1$, has the ensuing properties:
\end{theorem}
\begin{itemize}
\item [$(a)$] $K$ can be extended to the whole complex plane as an entire function.
\item [$(b)$] $K(1)=-\gamma/\alpha$.
\item [$(c)$] On the line $\Re e\: s=1$ the function $K$ and their derivatives have asymptotic behavior
\begin{equation}
\label{asympK}
K(1+it)= - \log\left|t\right| -\gamma -\frac{\pi i}{2} \operatorname*{sgn}(t) + O\left(\frac{1}{|t|^{\alpha}}\right)+O\left(\frac{1}{|t|^{\frac{\alpha}{2(\alpha-1)}}}\right),
\end{equation}
\begin{equation}
\label{asympK'}
K'(1+it)=A_{\alpha,1}|t|^{\frac{1-\alpha/2}{\alpha-1}}\exp\left(-i\operatorname*{sgn}(t)\left(B_{\alpha}|t|^{\frac{\alpha}{\alpha-1}}-\frac{\pi}{4}\right)\right)
+ O\left(\frac{1}{|t|}\right),
\end{equation}
and, for $m=2,3,\dots,$
\begin{equation}
\label{asympKderivatives}
K^{(m)}(1+it)=A_{\alpha,m}|t|^{\frac{m-\alpha/2}{\alpha-1}}\exp\left(-i\operatorname*{sgn}(t)\left(B_{\alpha}|t|^{\frac{\alpha}{\alpha-1}}-\frac{\pi}{4}\right)\right)+ O\left(|t|^{\frac{m-3\alpha/2}{\alpha-1}}\right)\: ,
\end{equation}
as $|t|\to\infty$, where
\begin{equation}
\label{asympKconstants}
B_{\alpha}= (\alpha-1)\alpha^{-\frac{\alpha}{\alpha-1}} \ \ \mbox{ and } \ \ A_{\alpha,m}=(-1)^{m}\alpha^{\frac{1/2-m}{\alpha-1}}
\sqrt{\frac{\pi }{2(\alpha-1)}}\: , \ \ \ m=1,2,3,\dots \:. 
\end{equation}
\end{itemize}
\begin{proof} 
We shall obtain all claimed properties of $K$ from those of the analytic function
$$F(z):= \operatorname*{F.p.} \int^{\infty}_{0} \frac{e^{-izu}}{u} e^{iu^{\alpha}}\mathrm{d}u\: , \  \ \ \Im m\:z<0\: .$$
The two functions are obviously linked via the relation
\begin{equation}
\label{asympeq0}
K(1+iz)= \frac{F(z) + \overline{F(-\bar{z})}}{2}\: .
\end{equation} 
We shall need a (continuous) Littlewood-Paley partition of unity \cite[Sect. 8.5]{hormander1997}. So, 
find an even smooth function $\varphi\in\mathcal{D}(\mathbb{R})$ with the following properties: $\operatorname*{supp}{\varphi}\subset(-1,1)$ and $\varphi(x)=1$ for $x\in[-1/2,1/2]$. Set $\psi(x)=-x\varphi'(x)$, an even test function with support on $(-1,1/2]\cup[1/2,1)$, so that we have the decomposition of the unity
$$
1=\varphi(x)+\int_{0}^{1}\psi(yx)\frac{\mathrm{d}y}{y}, \ \ \ x\in\mathbb{R}\: .
$$
This leads to the continuous Littlewood-Paley decomposition 
\begin{equation}
F(z)= \theta(z)+v(z)\: , \ \ \ \Im m\: z<0\: ,
\end{equation}
where 
\begin{equation*}
\theta(z)=\mathrm{F.p.} \int_{0}^{\infty} \frac{\varphi(u)e^{i(u^{\alpha}-z u)}}{u}\mathrm{d}u\: , \ \ \  v(z)= \int_{0}^{1}\Phi(y,z)\frac{\mathrm{d}y}{y}\: ,
\end{equation*}
and 
\begin{equation*}
\Phi(y,z)= \int_{0}^{\infty} \frac{\psi(yu)e^{i(u^{\alpha}-z u)}}{u}\mathrm{d}u
\: , \ \ \ \Im m\: z<0\: .
\end{equation*}
The formula for $v$ still makes sense for $z=t\in\mathbb{R}$ if it is interpreted in the sense of tempered distributions, where the integral with respect to $y$ is then understood as a weak integral in the space $\mathcal{S}'(\mathbb{R})$. Observe that $\theta(z)$ and $\Phi(y,z)$ are entire functions of $z$, as follows at once from the well-known Paley-Wiener-Schwartz theorem \cite{vladimirov}. The asymptotic behavior of $\theta$ and its derivatives on the real axis can be computed directly from the Estrada-Kanwal  generalization of Erd\'{e}lyi's asymptotic formula \cite[p. 148]{estrada-kanwal}; indeed, employing only one term from the quoted asymptotic formula, we obtain
\begin{align}
\label{asympeq1}
&\theta(t)= -\log\left|t\right| -\gamma -\frac{\pi i}{2} \operatorname*{sgn}(t) + O\left(\frac{1}{|t|^{\alpha}}\right) \ \ \mbox{and} 
\\
 &
\theta^{(m)}(t)= \frac{(-1)^{m}(m-1)!}{t^{m}} + O\left(\frac{1}{|t|^{\alpha+m}}\right),
\nonumber
\end{align}
$m=1,2,\dots$, as $|t|\to\infty$.
We now study the integral $\int_{0}^{1}\Phi(y,z)y^{-1}\mathrm{d}y$. If we consider $z=t+i\sigma$, we can write ($t\neq0$) 

\begin{equation*}
\partial^{m}_{z}\Phi(y,z)= y^{1-m}(-i)^{m}\left|t\right|^{\frac{1}{\alpha-1}}\int_{0}^{\infty}\rho_{m}\left(|t|^{\frac{1}{\alpha-1}}yx\right)e^{i|t|^{\frac{\alpha}{\alpha-1}}(x^{\alpha}- \operatorname*{sgn}(t) x)+\sigma|t|^{\frac{1}{\alpha-1}}x}\mathrm{d}x\: ,
\end{equation*}
where $\rho_{m}(x)=x^{m-1}\psi(x)$, $m\in\mathbb{N}$. We need to establish some asymptotic estimates for the integrals occurring in the above expression, namely, for 
\begin{equation}
\label{asympeq3}
J_{m}(y,t;\sigma)=\int_{0}^{\infty}\rho_{m}\left(|t|^{\frac{1}{\alpha-1}}yx\right)e^{i|t|^{\frac{\alpha}{\alpha-1}}(x^{\alpha}- \operatorname*{sgn}(t) x)+\sigma|t|^{\frac{1}{\alpha-1}}x}\mathrm{d}x\: .
\end{equation}
We shall show that for each $n\in\mathbb{N}$

\begin{equation}
\label{asympeq4}
J_{m}(y,t;\sigma)=\begin{cases} O\left(y^{n}t^{-n}\right)&\mbox{ if } t>0 \mbox{ and } t^{\frac{1}{\alpha-1}}y> 2\alpha^{\frac{1}{\alpha-1}}\: ,\\
O\left(y^{n\alpha-1}t^{-\frac{1}{\alpha-1}}\right) & \mbox{ if } t>0 \mbox{ and } t^{\frac{1}{\alpha-1}}y<1/2\: ,
\end{cases}
\end{equation}
and 
\begin{equation}
\label{asympeq5}
J_{m}(y,t;\sigma)=\begin{cases} O\left(y^{n}|t|^{-n}\right)& \mbox{ if } t<0 \mbox{ and }|t|^{\frac{1}{\alpha-1}}y\geq 1\: ,\\
O\left(y^{n\alpha-1}|t|^{-\frac{1}{\alpha-1}}\right) & \mbox{ if } t<0 \mbox{ and }|t|^{\frac{1}{\alpha-1}}y<1\: ,
\end{cases}
\end{equation}
where all big $O$-constants only depend on $\alpha$, $n$, and the $L^{\infty}$-norms of the derivatives of $\rho_{m}$. Notice that the estimates (\ref{asympeq4}) and (\ref{asympeq5}) yield, uniformly for $z$ in compacts of $\mathbb{C}$,
$$
\left|\partial^{m}_{z}\Phi(y,z)\right|= O_{n}(y^{n\alpha-m}) \ \ \ \mbox{if }y\left|t\right|^{\frac{1}{\alpha-1}}<1/2\: ,
$$
for any $n$, which proves that the integrals
$\int_{0}^{1}\partial^{m}_{z}\Phi(y,z)y^{-1}\mathrm{d}y$
are absolutely convergent in the space of entire functions and thus $v(z)$ is entire. In particular, we obtain that $F(z)$ is an entire function, which also implies that $K(s)$ is entire because of (\ref{asympeq0}). Furthermore, using (\ref{asympeq5}), one obtains at once that
\begin{equation}
\label{asympeq6}v^{(m)}(t)=O\left(|t|^{-n}\right)\ \  \ \mbox{as } t\to-\infty, \ \ \ \forall n\in\mathbb{N}\: .
\end{equation}

In order to prove (\ref{asympeq4}) in the range $t^{\frac{1}{\alpha-1}}y> 2\alpha^{\frac{1}{\alpha-1}}$, we rewrite (\ref{asympeq3})
as
$$
J_{m}(y,t;\sigma)= \int_{0}^{\infty}\frac{\rho_{m}(t^{\frac{1}{\alpha-1}}yx)}{{g'(x)}} g'(x)e^{i t^{\frac{\alpha}{\alpha-1}}g(x)}\mathrm{d}x\: , 
$$
where $g(x)=x^{\alpha}-x-i\sigma x/t$. The estimate (\ref{asympeq4}) for $t^{\frac{1}{\alpha-1}}y>2\alpha^{\frac{1}{\alpha-1}}$ follows by integrating by parts $n$ times and noticing that $\left|g'(x)\right|>1-2^{1-\alpha}>0$ for $x\in(0,\alpha^{\frac{-1}{\alpha-1}}/2)$. 
In fact, integrating by parts once gives
\begin{align*}
&J_{m}(y,t;\sigma)\leq 
\\
&
\quad \frac{\|g\|_{L^{\infty}}+\|g'\|_{L^{\infty}}}{(1-2^{1-\alpha})^{2}} t^{-\alpha/(\alpha-1)}\int_{0}^{(2\alpha^{\frac{1}{\alpha-1}})^{-1}}(yt^{\frac{1}{\alpha-1}}|\rho'_{m}(yt^{\frac{1}{\alpha-1}}x)|+|\rho_{m}(yt^{\frac{1}{\alpha-1}}x)|)\mathrm{d}x
\ll
yt^{-1},
\end{align*}
because $\rho(yt^{\frac{1}{\alpha-1}}x)$ vanishes for $x\geq (2\alpha^{\frac{1}{\alpha-1}})^{-1}$ and $ t^{-\alpha/(\alpha-1)}\leq (2\alpha^{\frac{1}{\alpha-1}})^{-1} yt^{-1}$. In the general case, we iterate this procedure $n$ times to obtain $J_{m}(y,t;\sigma)=O(y^{n}t^{-n})$ where the $O$-constant only depends on $\alpha$ and $\|\rho_{m}\|_{L^{\infty}(\mathbb{R})},$ $\|\rho'_{m}\|_{L^{\infty}(\mathbb{R})},$ $\dots,$ $\|\rho^{(n)}_{m}\|_{L^{\infty}(\mathbb{R})}$. On the other hand, if $t^{\frac{1}{\alpha-1}}y<1/2$, we integrate by parts $n$ times the integral written as
$$
J_{m}(y,t;\sigma)= \frac{1}{t^{\frac{1}{\alpha-1}}y}\int_{1/2}^{1}\frac{\rho_{m}(x)}{f'_{y}(x)} f'_{y}(x)e^{i{y^{-\alpha}}f_{y}(x)}\mathrm{d}x\: ,
$$
where $f_{y}(x)=x^{\alpha}-y^{\alpha-1}tx-i\sigma y^{\alpha-1}x$. The second part of (\ref{asympeq4}) holds because $\left|f'_{y}(x)\right|\geq \Re e\: f'_{y}(x)> (\alpha-1)2^{1-\alpha}$ and the derivatives of $f_{y}$ of order $\geq2$ are bounded on $(1/2,1)$; once again, the $O$-constant merely depends on $\alpha$ and $\|\rho_{m}\|_{L^{\infty}(\mathbb{R})},$ $\|\rho'_{m}\|_{L^{\infty}(\mathbb{R})},$ $\dots,$ $\|\rho^{(n)}_{m}\|_{L^{\infty}(\mathbb{R})}$. The estimate (\ref{asympeq5}) is proved in a similar fashion. 

We now obtain the asymptotic behavior of $v(t)$ and its derivatives as $t\to\infty$. Employing (\ref{asympeq4}), we have for each $n\in\mathbb{N}$
\begin{align*}
v^{(m)}(t)&= (-i)^{m}t^{\frac{1}{\alpha-1}}\int_{t^{-\frac{1}{\alpha-1}}/2}^{2 (\alpha/t)^{\frac{1}{\alpha-1}}} y^{-m}J_{m}(y,t;0)\mathrm{d}y +O\left(t^{-n}\right)
\\
&=
(-i)^{m}t^{\frac{m}{\alpha-1}}\int_{1/2}^{2 \alpha^{\frac{1}{\alpha-1}}} y^{-m}\int_{0}^{\infty}\rho_{m}(yx)e^{it^{\frac{\alpha}{\alpha-1}}(x^{\alpha}-x)}\mathrm{d}x\mathrm{d}y +O\left(t^{-n}\right)\: ,
\end{align*}
as $t\to\infty$.
The asymptotic expansion of $\int_{0}^{\infty}\rho_{m}(yx)e^{it^{\frac{\alpha}{\alpha-1}}(x^{\alpha}-x)}\mathrm{d}x$ can be derived as  a direct consequence of the stationary phase principle (cf. \cite[Thm. 7.7.5]{hormander1990}). The only critical point of $x^{\alpha}-x$ lies at $x=\alpha^{-\frac{1}{\alpha-1}}$, the stationary phase principle therefore leads, after a routine computation, to
$$
\int_{0}^{\infty}\rho_{m}(yx)e^{it^{\frac{\alpha}{\alpha-1}}(x^{\alpha}-x)}\mathrm{d}x= A_{\alpha}t^{-\frac{\alpha}{2(\alpha-1)}}e^{-i(\alpha-1)\left(\frac{t}{\alpha}\right)^{\frac{\alpha}{\alpha-1}}}\rho_{m}\left(\alpha^{\frac{1}{1-\alpha}}y\right)+O\left(t^{-\frac{3\alpha}{2(\alpha-1)}}\right)\: , 
$$
as $t\to\infty$, uniformly for $y\in (1/2,2\alpha^{\frac{1}{\alpha-1}})$, where 
$$
A_{\alpha}=\sqrt{\frac{2\pi i}{\alpha^{\frac{1}{\alpha-1}}\left(\alpha-1\right)}}
$$
and the big $O$-constant depends only on $\alpha$, $m$, and the derivatives of order $\leq 2$ of $\psi$. Observe also that
$$
\int_{1/2}^{2 \alpha^{\frac{1}{\alpha-1}}} y^{-m}\rho_{m}(\alpha^{\frac{1}{1-\alpha}}y)\mathrm{d}y=\alpha^{\frac{1-m}{\alpha-1}}\int_{1/2}^{1}\frac{\psi(y)}{y}\mathrm{d}y=\alpha^{\frac{1-m}{\alpha-1}}\: .
$$
Hence,
\begin{equation}
\label{asympeq7}
v^{(m)}(t)= (-i)^{m}\alpha^{\frac{1/2-m}{\alpha-1}}t^{\frac{m-\alpha/2}{\alpha-1}}e^{-i(\alpha-1)\left(\frac{t}{\alpha}\right)^{\frac{\alpha}{\alpha-1}}}
\sqrt{\frac{2\pi i}{\alpha-1}}+ O\left(t^{\frac{m-3\alpha/2}{\alpha-1}}\right)\: ,
\end{equation}
as $t\to\infty$. The asymptotic estimates (\ref{asympK})--(\ref{asympKderivatives}) with constants (\ref{asympKconstants}) follow by combining (\ref{asympeq0}), (\ref{asympeq1}), (\ref{asympeq6}), and (\ref{asympeq7}). Thus, the proofs of $(a)$ and $(c)$ are complete. It remains to establish the property $(b)$. Notice that $K(1)=\Re e\:F(0)$ because of (\ref{asympeq0}). On the other hand, applying the Cauchy theorem to 
$$
\oint_{\mathsf{C}}\frac{e^{i\xi^{\alpha}}}{\xi}\: \mathrm{d}\xi
$$ 
in the contours $\mathsf{C}=[\varepsilon,r]\cup\{\xi=re^{i\vartheta}: \vartheta\in [0,\pi/(2\alpha)]\} \cup\{\xi=xe^{i\frac{\pi}{2\alpha}}:\: x\in[\varepsilon,r]\} \cup\{\xi=\varepsilon e^{i\vartheta}: \vartheta\in [0,\pi/(2\alpha)]\},$ one deduces that
\begin{align*}
F(0)&=\mathrm{F.p.} \int^{\infty}_{0} \frac{e^{iu^{\alpha}}}{u}\: \mathrm{d}u
\\
&
=\mathrm{F.p.} \int^{\infty}_{0}\frac{e^{-x^{\alpha}}}{x}\: \mathrm{d}x+\lim_{\varepsilon\to0^{+}}i\int_{0}^{\frac{\pi}{2\alpha}}e^{i \varepsilon^{\alpha}e^{i\alpha \vartheta}}\mathrm{d}\vartheta
\\
&
= \frac{1}{\alpha}\mathrm{F.p.} \int^{\infty}_{0}\frac{e^{-x}}{x}\: \mathrm{d}x+\frac{i\pi}{2\alpha}
=-\frac{\gamma}{\alpha}+\frac{i\pi}{2\alpha}\:.
\end{align*}
\end{proof}

The previous theorem and (\ref{zetacontformula}) imply that $\zeta_{C}$ is analytic in $\mathbb{C}\setminus\{1\}$ and actually has a simple pole at $s=1$ with residue 
$$
\mathrm{Res}_{ s=1}\zeta_{C}(s)=e^{-\left(1-\frac{1}{\alpha}\right)\gamma}\:.
$$
 Thus, in view of part $(c)$ from Theorem \ref{thlacseries}, the function $N_{C}$ fulfills the hypotheses of Lemma \ref{lemsufcondcesaro} with $a=\exp(-\gamma(1-1/\alpha))$ for every $n$. Furthermore, (\ref{eqnvwkahane1}) is also satisfied, as $\zeta_{C}(s)-a/(s-1)$ is entire. Since we are interested in violating Kahane's condition, we must investigate (\ref{eqnvwkahane2}). The Leibniz rule for differentiation gives
\[ \left(\frac{\zeta_{C}(1+it)}{t}\right)' H(|t|-1) =\left(-\frac{e^{-K(1+it)-\gamma}  K'(1+it)}{t^{2}} + \frac{2e^{-K(1+it)-\gamma}}{t^{3}}\right)iH(|t|-1)\: .
\]
Using (\ref{asympK}) from Theorem \ref{thlacseries} we see that the absolute value of the second term is asymptotic to $(2/t^{2})H(|t|-1) \in L^{2}(\R)$. Employing Lemma \ref{lemnvwkahane} and (\ref{asympK}) once again, we find that Kahane's condition for $N_{C}$ becomes equivalent to
\[ \frac{K'(1+it)}{t}H(|t|-1)\in L^{2}(\R)\: .
\]
The asymptotic behavior of $t^{-1}K'(1+it)$ is given by (\ref{asympK'}):
\[ \frac{K'(1+it)}{t} = A_{\alpha,1}|t|^{\frac{2-3\alpha/2}{\alpha-1}}\exp\left(-i\operatorname*{sgn}(t)\left(B_{\alpha}|t|^{\frac{\alpha}{\alpha-1}}-\frac{\pi}{4}\right)\right)
+ O\left(\frac{1}{|t|^{2}}\right), \quad |t| \rightarrow \infty\: .
\]
The second term above is $L^{2}$ for $|t|\geq1$, whereas the first term is $L^2$ only for $\alpha>3/2$. We summarize our results in the following proposition, which shows that our continuous number system satisfies the properties stated in Theorem \ref{thmain}. As usual, we set
$$
\operatorname*{Li}(x)= \int_{2}^{x}\frac{\mathrm{d}t}{\log t}\: .
$$
\begin{proposition}
\label{propcontsystem} Let $\alpha>1$. The functions $N_{C}$ and $\Pi_{C}$ satisfy
$$
N_{C}(x)=xe^{-\gamma\left(1-\frac{1}{\alpha}\right)}+ O\left(\frac{x}{\log^{n}x}\right) \quad (\mathrm{C}), \quad \mbox{for }n=1,2,\dots\: ,
$$
and 
\begin{equation}
\label{asympPIcont} \Pi_{C}(x)= \operatorname*{Li}(x)+ O\left(\frac{x}{\log^{\alpha}x}\right)\: .
\end{equation}
 One has 
$$
\int_{1}^{\infty}\left|\frac{\left(N_{C}(x)-xe^{-\gamma\left(1-\frac{1}{\alpha}\right)}\right)\log x}{x}\right|^{2}\frac{\mathrm{d}x}{x}=\infty 
$$
if and only if $1<\alpha\leq 3/2$.
\end{proposition}
\begin{proof} We only need to prove (\ref{asympPIcont}). This follows from a calculation,
\begin{align*}
\Pi_{C}(x)-\operatorname*{Li}(x) &= - \int^{x}_{2} \frac{\cos(\log^{\alpha}u)}{\log u}\mathrm{d}u + O(1)
\\
& = - \frac{1}{\alpha}\int^{x}_{2} \frac{u}{\log^{\alpha}u}\:\mathrm{d}(\sin(\log^{\alpha}u))+ O(1)
\\ 
& 
=  \frac{1}{\alpha} \int^{x}_{2} \frac{\sin(\log^{\alpha}u) }{\log^{\alpha}u}\:\mathrm{d}u - \int^{ x}_{ 2} \frac{\sin(\log^{\alpha}u) }{\log^{\alpha+1}u}\:\mathrm{d}u+ O\left(\frac{x}{\log^{\alpha }x}\right)
\\
&
= O\left(\frac{x}{\log^{\alpha }x}\right)\:,
\end{align*}
because
\begin{align*}
\int^{x}_{2} \frac{\sin(\log^{\alpha}u) }{\log^{\alpha}u}\:\mathrm{d}u\ll \int^{x}_{\sqrt{ x}}  \frac{\mathrm{d}u}{\log^{\alpha}u} +O(\sqrt{x})\ll\frac{x}{\log^{\alpha }x}\: ,
\end{align*}
and similarly the second integral has growth order $\ll x/\log^{\alpha+1}x$.
\end{proof}
\section{Discrete examples: Proof of Theorem \ref{thmain}}
\label{the discrete example}
We now discretize the family of continuous examples from the previous section. Let $\alpha>1$. We recall the functions of the continuous example were
\[ \Pi_{C}(x) = \int^{x}_{1} \frac{1 - \cos(\log^{\alpha} u)}{\log u}\: \mathrm{d}u \quad \mbox{ and } \quad \zeta_{C}(s) = \frac{e^{-\gamma}e^{-K(s)}}{s-1}\: ,
\]
where $K$ is the entire function studied in Theorem \ref{thlacseries}. Our set of generalized primes $P_{\alpha}$ is defined as in the introduction, namely, its $r$-th prime $p_{r}$ is $\Pi_{C}^{-1}(r)$. 

We shall now establish Theorem \ref{thmain} for $P_{\alpha}$. Throughout this section $\pi$, $\zeta$, $N$, and $\Pi$ (cf. (\ref{defriemann})) stand for the number-theoretic functions associated to $P_{\alpha}$. We choose to omit the subscript $P_{\alpha}$ not to overload the notation.  As an easy consequence of the definition we obtain the inequality $0\leq \Pi_{C}(x) - \pi(x) \leq 1$. By combining this observation with (\ref{asympPIcont}) from Proposition \ref{propcontsystem}, we obtain at once that $\pi$ satisfies the PNT
\begin{equation}
\label{discreteeq1} \pi(x)=\operatorname*{Li}(x)+ O\left(\frac{x}{\log^{\alpha}x}\right),
\end{equation}
where the only requirement is $\alpha>1$.

This shows that the asymptotic formula (\ref{ibpneq7}) from Theorem \ref{thmain} holds for $1<\alpha\leq 2$. Naturally, (\ref{discreteeq1}) implies that our set of generalized primes $P_{\alpha}$ satisfies a version of Mertens' second theorem, which we state in the next lemma because we shall need it below. The proof is a simple application of integration by parts, the relation $\pi(x)=\Pi_{C}(x)+O(1)$, and the explicit formula for $\Pi_{C}$; we therefore omit it. Notice that the asymptotic estimate is even valid for $0<\alpha\leq 1$, with the obvious extension of the definition of $P_{\alpha}$ for these parameters.
\begin{lemma} \label{correcprimes} Let $\alpha>0$. The generalized prime number system $P_{\alpha}$ satisfies the following Mertens type asymptotic estimate
\[ \sum_{p_{r} \leq x} \frac{1}{p_{r}} = \log \log x + M + O\left(\frac{1}{\log^{\alpha}x}\right)\: .
\]
for some constant $M=M_{\alpha}$.
\end{lemma}

We now concentrate in showing (\ref{Cesaroeqgi}) and (\ref{nkahane}). We will prove that they hold with the constant
\begin{equation}
\label{discreteconstant} a_{\alpha}= \exp\left(-\gamma\left(1-\frac{1}{\alpha}\right) + \int^{\infty}_{1}x^{-1}\mathrm{d}(\Pi - \Pi_{C})(x)\right)\: .
\end{equation}
 We express the zeta function of this prime number system in terms of $\zeta_{C}$. We find
\begin{equation}\label{formulazetadiscrete}\zeta(s) = \zeta_{C}(s)\exp\left(\int^{\infty}_{1} x^{-s} \mathrm{d}(\Pi-\Pi_{C})(x)\right)\: .
\end{equation}
Note that $\int^{\infty}_{1} x^{-s} \mathrm{d}(\Pi-\Pi_{C})(x)$ is analytic on the half-plane $\Re e\:s>1/2$ because $\Pi(x)-\Pi_{C}(x)=\Pi(x)-\pi(x)+\pi(x)-\Pi_{C}(x)=O(x^{1/2})+O(1)$. Employing Theorem \ref{thlacseries}, we see that, when $\alpha>1$, $\zeta$ is also analytic in $\Re e\:s>1/2$ except at $s=1$ and
$$
\mathrm{Res}_{ s=1}\zeta(s)=a_{\alpha}\:,
$$
where $a_{\alpha}$ is given by (\ref{discreteconstant}). Hence, the hypothesis (\ref{eqnvwkahane1}) from Lemma \ref{lemnvwkahane} is satisfied with $a_{\alpha}$ for all $\alpha>1$. We also mention the set of generalized primes $P_{\alpha}$ satisfies the Riemann hypothesis in the form: $\zeta(s)\neq0$ for $\Re e\:s>1/2$, $s\neq 1$. (This follows from the factorizations (\ref{formulazetadiscrete}), (\ref{zetacontformula}), and Part $(a)$ of Theorem \ref{thlacseries}.)

As we are interested in the growth behavior of $\zeta$ on the line \mbox{$\Re e$ $s = 1$}, we will try to control the term $\int^{\infty}_{1} x^{-1-it}\mathrm{d}(\Pi-\Pi_{C})(x)$. The following lemma gives a useful bound for it and this section will mostly be dedicated to its proof.

\begin{lemma} \label{lembegrenzingzeta} Let $\alpha\geq 1$. The discrete prime number system $P_{\alpha}$ satisfies the following bound:
\[ \left|\Re e \int^{\infty}_{1} x^{-1-it}\mathrm{d}(\Pi-\Pi_{C})(x)\right| = \left|\int^{\infty}_{1} \frac{\cos(t\log x)}{x}\mathrm{d}(\Pi-\Pi_{C})(x)\right| =O(\log \log |t|)\: .
\]
\end{lemma}
The same bound holds for the imaginary part and the proof is exactly the same. We first give a Hoheisel-Ingham type estimate for the gaps between consecutive primes from $P_{\alpha}$.

\begin{lemma}\label{lemafschattingpriemafstand}
Let $\alpha\geq 1$. Then,  we have the bound $p_{r+1}-p_r<p_r^{2/3}\log p_r$ for sufficiently large $r$.
\end{lemma}
\begin{proof}
Set $d=p_r^{2/3}\log p_r$. It suffices to show that for $p_r$ sufficiently large we have
\[
\int_{p_r}^{p_r+d} \frac{1-\cos(\log^\alpha u)}{\log u}\;\mathrm{d}u > 1\:, 
\]
which is certainly implied by $\int_{p_r}^{p_r+d}(1-\cos(\log^\alpha u))\;\mathrm{d}u>2\log p_r$. If $p_r<u<p_r+d$, then 
\[
\log^\alpha \left(u+\frac{d}{4}\right)-\log^\alpha u\geq \frac{\alpha d\log^{\alpha-1}u}{4(u+\frac{d}{4})}\geq\frac{d}{5p_r}\:.
\]
Since $\cos t\leq 1-t^{2}/3$ for $|t|<\pi/4$, this implies that among the four intervals $[p_r, p_r+d/4]$, \ldots, $[p_r+3d/4, p_r+d]$ there is one, which we call $I$, such that
\[
\cos( \log^\alpha u)\leq 1-\frac{d^2}{75p_r^2}
\]
for all $u\in I$. The integrand in question is non-negative for all $u$, we may thus restrict the range of integration to $I$ and obtain as lower bound
\[
 \int_{I}(1-\cos(\log^\alpha u))\;\mathrm{d}u > \frac{d}{4}\cdot \frac{d^2}{75p_r^2} = \frac{\log^3 p_r}{300}>2\log p_r\:.
\]
Hence our claim follows.
\end{proof}

We can now give a proof of Lemma \ref{lembegrenzingzeta}.

{\noindent\textit{Proof of Lemma \ref{lembegrenzingzeta}.} }
First we are going to change the measure we integrate by,
\begin{align*} \left|\int^{\infty}_{1} \frac{\cos(t\log x)}{x}\:\mathrm{d}(\Pi-\Pi_{C})(x)\right| & \leq \left|\int^{\infty}_{1} \frac{\cos(t\log x)}{x}\: \mathrm{d}(\Pi-\pi)(x)\right| \\
& \quad + \left|\int^{\infty}_{1} \frac{\cos(t\log x)}{x}\: \mathrm{d}(\pi-\Pi_{C})(x)\right|\: .
\end{align*}
We can estimate the first integral as follows:
\[ \left|\int^{\infty}_{1} \frac{\cos(t\log x)}{x}\:\mathrm{d}(\Pi-\pi)(x)\right| \leq \int^{\infty}_{1} \frac{1}{x}\:\mathrm{d}(\Pi-\pi)(x) < \infty\: ,
\]
where we have used that $d(\Pi-\pi)$ is a positive measure and $\Pi(x) - \pi(x) = O(x^{1/2})$. Only the second integral remains to be estimated. We are going to split the integral in intervals of the form $[p_{r},p_{r+1})$. Such an interval delivers the contribution
\[ \left|\int_{[p_{r},p_{r+1})} \frac{\cos(t\log x)}{x}\mathrm{d}(\pi-\Pi_{C})(x)\right| = \left|\int^{p_{r+1}}_{p_{r}}\left(\frac{\cos(t\log p_{r})}{p_{r}} - \frac{\cos(t\log x)}{x}\right)\:\mathrm{d}\Pi_{C}(x)\right|\: ,
\]
since $\int^{p_{r+1}}_{p_{r}}\mathrm{d}\Pi_{C}(x) = 1$. This integral can be further estimated by
\begin{align*}
& \left|\int^{p_{r+1}}_{p_{r}}\left(\frac{\cos(t\log p_{r})}{p_{r}} - \frac{\cos(t\log x)}{x}\right)\mathrm{d}\Pi_{C}(x)\right|\leq
\\
&
\int^{p_{r+1}}_{p_{r}}\left|\frac{\cos(t\log p_{r})}{p_{r}} - \frac{\cos(t\log x)}{p_{r}}\right|\mathrm{d}\Pi_{C}(x)+ \int^{p_{r+1}}_{p_{r}}\left|\frac{\cos(t\log x)}{p_{r}} - \frac{\cos(t\log x)}{x}\right|\mathrm{d}\Pi_{C}(x)\: .
\end{align*}
The second of these integrals can be bounded by
\[\int^{p_{r+1}}_{p_{r}}\left(\frac{1}{p_{r}} - \frac{1}{p_{r+1}}\right)\mathrm{d}\Pi_{C}(x) = \frac{p_{r+1} - p_{r}}{p_{r}p_{r+1}} \leq \frac{p_{r}^{2/3 + \varepsilon}}{p_{r}^{2}}\: ,
\]
by Lemma \ref{lemafschattingpriemafstand}, and after summation on $r$ this gives a contribution which is finite and does not depend on $t$. We now bound the other integral. By the mean value theorem, we have
\begin{align*}
\int^{p_{r+1}}_{p_{r}}\left|\frac{\cos(t\log p_{r})}{p_{r}} - \frac{\cos(t\log x)}{p_{r}}\right|\mathrm{d}\Pi_{C}(x) & \leq \frac{|t \log p_{r+1} - t \log p_{r} |}{p_{r}}\\
& \leq \frac{|t|}{p_{r}} \log\left(1 +  \frac{p_{r}^{2/3 + \varepsilon}}{p_{r}}\right)\\
& \leq \frac{|t|}{p_{r}^{4/3-\varepsilon}} \leq \frac{1}{p_{r}^{5/4}}
\end{align*}
for $p_{r} \geq| t|^{13}$ and $p_r$ sufficiently large. As the sum over finitely many small $p_r$ is $O(1)$, the latter condition is insubstantial. After summation on $r$ we see that these integrals deliver a finite contribution which does not depend on $t$. Finally, it remains to bound the integrals for $p_{r} \leq |t|^{13}$. We can estimate these as follows because of Corollary \ref{correcprimes}:
\[ \sum_{p_{r}\leq|t|^{13}} \int^{p_{r+1}}_{p_{r}}\left|\frac{\cos(t\log p_{r})}{p_{r}} - \frac{\cos(t\log x)}{p_{r}}\right|\mathrm{d}\Pi_{C}(x) \leq \sum_{p_{r}\leq |t|^{13}} \frac{2}{p_{r}} = O(\log \log |t|)\: .
\]
{\hfill$\square$\par\bigskip}

With the same techniques the following bounds can also be established:
\begin{equation}\label{estimatezetaderivatives} \int^{\infty}_{1} x^{-1-it} \log^{n} x \:\mathrm{d}(\Pi-\Pi_{C})(x)= O(\log^{n} |t|), \quad n=1,2,3,\dots\: .
\end{equation}

We have set the ground for the remaining part of the proof of Theorem \ref{thmain}.
With these bounds it is clear that $\zeta(1+it),\zeta'(1+it)$, $\zeta''(1+it)$, \dots have at most polynomial growth. By Lemma \ref{lemsufcondcesaro} the counting function $N$ of this discrete prime number system satisfies the Ces\`{a}ro behavior (\ref{Cesaroeqgi}) with the constant (\ref{discreteconstant}) whenever $\alpha>1$.
For Kaha\-ne's condition we calculate $(\zeta(1+it)t^{-1})'$ by the Leibniz rule. All the involved terms are $L^{2}$ except possibly for
\begin{equation}
\label{finalkahane} \frac{e^{-K(1+it)} K'(1+it) \exp\left(\int^{\infty}_{1} x^{-1-it} \mathrm{d}(\Pi-\Pi_{C})(x)\right)}{t^{2}}\ .
\end{equation}
Using the fact that there exists an $m \in \mathbb{N}$ such that\footnote{The proof of Lemma \ref{lembegrenzingzeta} shows that $m=2$ suffices.}
$$\left|\exp\left(\int^{\infty}_{1} x^{-1-it} \mathrm{d}(\Pi-\Pi_{C})(x)\right)\right| \gg\frac{1}{\log^{m}|t|} \quad \mbox{for } |t|\gg1\:,$$
 and applying Theorem \ref{thlacseries}, exactly as in the discussion from Section \ref{the continuous example}, we find that (\ref{finalkahane}) is not $L^{2}$ when $1<\alpha < 3/2$. Lemma \ref{lemnvwkahane} yields (\ref{nkahane}) for $1<\alpha<3/2$ and Theorem \ref{thmain} has been so established for $P_{\alpha}$.

\begin{remark} If $\alpha>3/2$ then $P_{\alpha}$ does satisfy Kahane's condition, as also follows from the above argument. In contrast to Proposition \ref{propcontsystem}, whether Kahane's condition holds true or false for $P_{3/2}$ is an open question.
\end{remark}
\section{On the examples of Diamond and Beurling. Proof of Theorem \ref{thDiamond}}
\label{Diamond example}

In the previous section we extracted a discrete example from a continuous one by applying Diamond's discretization procedure used in \cite{diamond2} to show the sharpness of Beurling's PNT. However, our technique used to prove that our family of discrete examples have the desired properties from Theorem \ref{thmain} was quite different (Diamond's technique is rather based on operational calculus for the multiplicative convolution of measures). In this section we show how our method can also be applied to provide an alternative analysis of the Diamond-Beurling examples for the sharpness of the condition $\gamma=3/2$ in Beurling's theorem. In fact, our techniques below leads to a more precise asymptotic formula for the generalized integer counting function of Diamond's example. So, the goal of this section is to prove Theorem \ref{thDiamond}.
\par
We recall that Beurling's example provided in \cite{beurling} is the Riemann prime counting function
\[ \Pi_{C,1}(x) = \int^{x}_{1} \frac{1 - \cos(\log u)}{\log u}\mathrm{d}u\: ,  
\]
corresponding to the case $\alpha=1$ in (\ref{Riemannconteq}). Its associated zeta function is
\[ \zeta_{C,1}(s) := \left( 1 + \frac{1}{(s-1)^{2}}\right)^{1/2}=\exp\left(\int_{1}^{\infty}x^{-s}\mathrm{d}\Pi_{C,1}(x)\right).
\]
Diamond's example $P_{1}$ is then the case $\alpha=1$ of (\ref{eqdiscrete}). We immediately get 
\begin{equation*}
\Pi_{C,1}(x)=  \frac{x}{\log x}\left(1-\frac{\sqrt{2}}{2}\cos \left(\log x-\frac{\pi}{4}\right)\right) +O\left(\frac{x}{\log^{2} x}\right)
\end{equation*}
and, since $\pi_{P_{1}}(x)=\Pi_{C,1}(x)+O(1)$,
\begin{equation}
\label{asympdiamondexeq1}
\pi_{P_{1}}(x)=\frac{x}{\log x}\left(1-\frac{\sqrt{2}}{2}\cos \left(\log x-\frac{\pi}{4}\right)\right)  +O\left(\frac{x}{\log^{2} x}\right)\: ,
\end{equation}
whence neither $\Pi_{C,1}$ nor $\pi_{P_{1}}$ satisfy the PNT.

To study $N_{C,1}$ and $N_{P_{1}}$, we need a number of properties of their zeta functions on $\Re e\: s=1$. We control $\zeta_{C,1}$ completely. On this line $\zeta_{C,1}$ is analytic except for a simple pole at $s=1$ with residue 1, and two branch singularities at $s=1+i$ and $s=1-i$, where $\zeta_{C,1}$ is still continuous. Writing $\zeta_{C,1}(s)=(s-1-i)^{1/2}(s-1+i)^{1/2}(s-1)^{-1}$, we have  around $1\pm i$ the expansions
\begin{equation}
\label{diamondexeq2}
\zeta_{C,1}(s) = (1-i)(s-1- i)^{1/2}+ \sum_{k=1}^{\infty} a_{k} (s-1- i)^{k+1/2}\: , \quad |s-1- i|<1\: .
\end{equation}
and
\begin{equation}
\label{diamondexeq3}
\zeta_{C,1}(s) = (1+i)(s-1+ i)^{1/2}+ \sum_{k=1}^{\infty} \overline{a}_{k} (s-1+ i)^{k+1/2}\: , \quad |s-1+ i|<1\: ,
\end{equation}
where explicitly $a_{k}=(1-i)i^{k}\sum_{j=0}^{k}\binom{1/2}{j}(-1/2)^{j}$. On the other hand, $\int_{1}^{\infty} x^{-s}\mathrm{d}(\Pi_{P_{1}}-\Pi_{C,1})(x)$ is analytic on the half-plane $\Re e\:s>1/2$, where $\Pi_{P_{1}}$ is the Riemann generalized prime counting function associated to $P_{1}$. So, 
\begin{equation}
\label{diamondexeq4}
\zeta_{P_{1}}(s)= \left( 1 + \frac{1}{(s-1)^{2}}\right)^{1/2}\exp\left(\int_{1}^{\infty} x^{-s}\mathrm{d}(\Pi_{P_{1}}-\Pi_{C,1})(x)\right)\: , 
\end{equation}
and we obtain that $\zeta_{P_{1}}$ shares similar analytic properties as those of $\zeta_{C,1}$, namely, it has a simple pole at $s=1$, with residue
\begin{equation}
\label{diamondexeq5}
c:=\mathrm{Res}_{s=1}\zeta_{P_{1}}(s)=\exp\left(\int_{1}^{\infty} x^{-1}\mathrm{d}(\Pi_{P_{1}}-\Pi_{C,1})(x)\right)>0\: , 
\end{equation}
and two branch singularities at $s=1\pm i$. We also have the expansions at $s=1\pm i$
\begin{equation}
\label{diamondexeq6}
\zeta_{P_1}(s) = b_{0}(s-1-i)^{1/2}+ \sum_{k=1}^{\infty} b_{k} (s-1- i)^{k+1/2}\: , \quad |s-1-i|<1/2\: ,
\end{equation}
and 
\begin{equation}
\label{diamondexeq7}
\zeta_{P_1}(s) = \overline{b}_{0}(s-1+i)^{1/2}+ \sum_{k=1}^{\infty} \overline{b}_{k} (s-1+ i)^{k+1/2}\: , \quad |s-1+i|<1/2\: ,
\end{equation}
where $b_{0}=(1-i)\exp\left(\int_{1}^{\infty}x^{-1-i}\mathrm{d}(\Pi_{P_{1}}-\Pi_{C,1})(x)\right)\neq 0$ and the rest of the constants $b_{j}$ come from (\ref{diamondexeq2}) and the Taylor expansion of $\exp\left(\int_{0}^{\infty}x^{-s}\mathrm{d}(\Pi_{P_{1}}-\Pi_{C,1})(x)\right)$ at $s=1+i$.

We shall deduce full asymptotic series for $N_{P_{1}}(x)$ and $N_{C,1}(x)$ simultaneously from the ensuing general result.

\begin{theorem} \label{thremaindertauberian} Let $N$ be non-decreasing and vanishing for $x\leq 1$ with zeta function $\zeta(s)=\int_{1^{-}}^{\infty}x^{-s}dN(x)$ convergent on $\Re e\: s>1$. Suppose there are constants $a, r_{1}, \dots,r_{n}\in[0,\infty)$ and $\theta_{1},\dots,\theta_{n}\in[0,2\pi)$ such that
\begin{equation}\label{remaindereq1} G(s) := \zeta(s) - \frac{a}{s-1} - s\sum^{n}_{j=1}\left( r_{j}e^{\theta_{j} i} (s-1-i)^{j -\frac{1}{2}} + r_{j}e^{-\theta_{j}i} (s-1+i)^{j -\frac{1}{2}}\right)
\end{equation}
admits a $C^{n}$-extension to the line \mbox{$\Re e$ $s = 1$} and 
\[ \left|G^{(j)}(1+it)\right| = O(\left|t\right|^{\beta + n - j})\:, \quad |t|\to\infty\:, \quad j=0,1,\dots,n\: ,
\]
for $\beta \geq 0$.
Then 
\[ N(x) = ax + \frac{2x}{\log^{1/2} x}\sum^{n}_{j=1} \frac{r_{j}\cos(\log x+\theta_{j})}{\Gamma(-j+1/2)\log^{j}x}  + O\left(\frac{x}{\log^{\frac{n}{1+\beta}} x}\right), \quad x \rightarrow \infty\: ,
\]
\end{theorem}

\begin{proof} Set 
\[T(x):= ae^{x} + 2e^{x}\sum^{n}_{j=1}\left(r_{j}\cos(\theta_{j})\cos(x) - r_{j}\sin(\theta_{j})\sin(x)\right)\frac{x_{+}^{-j -\frac{1}{2}}}{\Gamma(-j+1/2)}
\] and define $R(x) := e^{-x}(N(e^{x}) - T(x))$. The tempered distributions $x_{+}^{-j-1/2}$ are those defined in \cite[Sect. 2.4]{estrada-kanwal}, i.e., the extension to $[0,\infty)$ of the singular functions $x^{-j-1/2}H(x)$ at $x=0$ via Hadamard finite part regularization. By the classical Wiener-Ikehara theorem we have that $N(x) \sim ax$ and this implies $R(x) = o(1)$. We have to show that $R(x) = O(x^{-n/(1+\beta)})$ as $x\to\infty$. Since $\mathcal{L}\{\cos(x)x^{-j-1/2}_{+};s\} = (\Gamma(-j+1/2)/2)[(s-i)^{j-1/2} + (s+i)^{j-1/2}]$ and $\mathcal{L}\{\sin(x)x^{-j-1/2}_{+};s\} = (\Gamma(-j+1/2)/(2i))[(s-i)^{j-1/2} - (s+i)^{j-1/2}]$, we have
$s \mathcal{L}\{R;s-1\} = G(s) - a$
Letting \mbox{$\Re e$ $s \rightarrow 1^{+}$}, we obtain that $\hat{R}(t) = (1+it)^{-1}(G(1+it) - a)$ in the space $\mathcal{S}'(\mathbb{R})$.

 We now derive a useful relation for $R$. Notice that there exists a $B$ such that $\left|T'(x)\right| \leq Be^{x}$ for $x\geq 1$. Applying the mean value theorem to $T$ and using the fact that $N$ is non-decreasing, we obtain 
\begin{equation*} R(y) \geq \frac{N(e^{x}) - T(x)}{e^{x}} \frac{e^{x}}{e^{y}} - B(y-x) \geq \frac{R(x)}{4} 
\end{equation*}
if $x \leq y \leq x+ \min\{R(x)/2B,\log(4/3)\}$ and $R(x) > 0$. Similarly, we have
\begin{equation*} -R(y) \geq -\frac{R(x)}{2} \quad \quad \mbox{if }R(x) < 0 \mbox{ and }x+\frac{R(x)}{2B} \leq y \leq x\:. 
\end{equation*}

 We now estimate $R$ if $R(x) > 0$. The case $R(x) < 0$ can be treated similarly. We choose an $\varepsilon \leq \min\{R(x)/2B,\log(4/3)\}$ and a test function $\phi \in \mathcal{D}(0,1)$ such that $\phi \geq 0$ and $\int_{-\infty}^{\infty} \phi(y)\mathrm{d}y = 1$. Using the derived inequality for  $R$ and the estimates on the derivatives of $G$, we obtain
\begin{align*}
 R(x) &\leq \frac{4}{\varepsilon} \int^{\varepsilon}_{0} R(y+x) \phi\left(\frac{y}{\varepsilon}\right) \mathrm{d}y \\
& = \frac{2}{\pi} \int^{\infty}_{-\infty} \hat{R}(t) e^{ixt} \hat{\phi}(-\varepsilon t)\: \mathrm{d}t\\
& = \frac{2}{(ix)^{n}\pi} \int^{\infty}_{-\infty} e^{ixt} \left(\hat{R}(t) \hat{\phi}(-\varepsilon t)\right)^{(n)} \mathrm{d}t\\
& =  O(1)x^{-n} \sum_{j=0}^{n}{n \choose j} \int^{\infty}_{-\infty} (1 + \left|t\right|)^{\beta -1 + n - j} \varepsilon^{n-j} |\hat{\phi}^{(n-j)}(-\varepsilon t)|\: \mathrm{d}t\\
&= O(1) x^{-n} \varepsilon^{-\beta}\: ,
\end{align*}
where we have used Parseval's relation in the distributional sense. If we choose\footnote{Since $R(x) = o(1)$, we may assume that $R(x)/2B \leq \log(4/3)$ for $x$ large enough.} $\varepsilon = R(x)/2B$, we get that $R(x) = O(x^{-n/(1+\beta)})$. A similar reasoning gives the result for $R(x) < 0$. This concludes the proof of the theorem.
\end{proof}

We can apply this theorem directly to $N_{C}$. Indeed, employing (\ref{diamondexeq2}) and (\ref{diamondexeq3}), one concludes that
\begin{align}
\label{diamondexeq8} N_{C}(x)&\sim x-\frac{x\sin(\log x)}{\sqrt{\pi}\log^{3/2} x}+ \frac{x}{\log^{5/2}x}\sum^{\infty}_{j=0}c_{j}\frac{\cos(\log x+\vartheta_{j})}{\log^{j}x}\\
\nonumber
&= x-\frac{x\sin(\log x)}{\sqrt{\pi}\log^{3/2} x}+ O\left(\frac{x}{\log^{5/2}x}\right), \quad x\to\infty\: ,
\end{align}
for some constants $c_{j}$ and $\vartheta_{j}$.

To show that $N_{P_{1}}$ has a similar asymptotic series, we need to look at the growth of $\zeta_{P_{1}}$ on $\Re e\:s=1$. This can be achieved with the aid of Lemma \ref{lembegrenzingzeta} and the bounds (\ref{estimatezetaderivatives}). In fact, if we combine those estimates with the formula (\ref{diamondexeq4}), we obtain at once that $\zeta_{P_{1}}^{(n)}(1+it)= O(\log^{n+2}|t| )$ for $|t|>2.$ This and the expansions (\ref{diamondexeq6}) and (\ref{diamondexeq7}) allow us to apply Theorem \ref{thremaindertauberian} and conclude that $N_{P_{1}}(x)$ has an asymptotic series (\ref{diamondexasympexp}) as $x\to\infty$, where the constant $c$ is given by (\ref{diamondexeq5}),
$$
d_{0}=\frac{1}{\sqrt\pi} \exp\left(\int_{1}^{\infty}\frac{\cos(\log x)}{x}\:\mathrm{d}(\Pi_{P_{1}}-\Pi_{C,1})(x)\right)>0
$$ 
and
$$
\theta_{0}=\frac{\pi}{2}-\int_{1}^{\infty}\frac{\sin(\log x)}{x}\:\mathrm{d}(\Pi_{P_{1}}-\Pi_{C,1})(x)\: .
$$
The proof of Theorem \ref{thDiamond} is complete.

We conclude this section with a remark:
\begin{remark}

The asymptotic formula $N_{C,1}(x)=x+O(x/\log ^{3/2}x)$ was first obtained by Beurling \cite{beurling} via the Perron inversion formula and contour integration. The asymptotic expansion (\ref{diamondexeq8}) appears already in Diamond's paper \cite{diamond2}. He refined Beurling's computation and also deduced from (\ref{diamondexeq8}) the first order approximation $N_{P_1}(x)=cx+O(x/\log ^{3/2}x)$ via convolution techniques. On the other hand, the asymptotic formula (\ref{diamondexasympexp}) is new and our proof, in contrast to those of Diamond and Beurling, avoids any use of information about the zeta functions on the region $\Re e \:s<1$. 
 
\end{remark}

\section{Proof of Theorem \ref{propExample}}
\label{SecExample}
In this section we amend the arguments from \cite{s-v} and show that the number system constructed in \cite[Sect. 6]{s-v} does satisfy the requirements from Theorem \ref{propExample}. This generalized prime number system is denoted here by $P^{\ast}$ and is constructed by removing and doubling suitable blocks of ordinary rational primes. Throughout this section we write $\pi=\pi_{P^\ast}$ and $N=N_{P^{\ast}}$, once again to avoid an unnecessary overloading in the notation. For the sake of completeness, some parts of this section overlap with \cite{s-v}. What differs here from \cite[Sect. 6]{s-v} is the crucial \cite[Lemma 6.3]{s-v} and the proof of \cite[Prop. 6.2]{s-v}, which substantially require new technical work.  

For the construction of our set of generalized primes, we begin by selecting a sequence of integers $x_i$, where $x_1$ is chosen so large that for all $x>x_1$ the interval $[x, x+\frac{x}{\log^{1/3} x}]$ contains more
than $\frac{x}{2\log^{4/3} x}$ ordinary rational prime numbers and $x_{i+1}=\lfloor2^{\sqrt[4]{x_i}
}\rfloor$. One has that $i=O(\log\log x_{i})$ and we may thus assume that $i\leq\log^{1/6} x_{i}$. We associate to each $x_i$ four disjoint intervals
$I_{i, 1}, \ldots, I_{i, 4}$. We start with $I_{i, 2}=[x_i,
  x_i+\frac{x_i}{\log^{1/3} x_i}]$ and define $I_{i,3}$ as the contiguous interval starting
at $x_{i}+\frac{x_i}{\log^{1/3} x_i}$ which contains as many (ordinary rational) prime numbers
as $I_{i, 2}$. It is important to notice that each of the intervals $I_{i,2}$ and $I_{i,3}$ has at least $\frac{x_{i}}{2\log ^{4/3}x_{i}}$ ordinary rational prime numbers. Therefore, the length of $I_{i,3}$ is also at most $O(\frac{x_i}{\log^{1/3}x_i})$, in view of the classical PNT. We now choose $I_{i, 1}$ and $I_{i,4}$ in such a
way that they fulfill the properties of following lemma, whose proof was given in \cite{s-v}.

\begin{lemma}
\label{bpexampleclaim}
There are intervals $I_{i, 1}$ and $I_{i, 4}$ such that $I_{i,1}$ has upper bound $x_i$, $I_{i, 4}$ has lower bound
equal to the upper bound of $I_{i, 3}$, and $I_{i, 1}$ and $I_{i, 4}$ contain
the same number of (ordinary rational) primes, and
\[
\prod_{\nu=1}^{i} \prod_{p\in I_{\nu,1}\cup I_{\nu,3}}\left(1-\frac{1}{p}\right)^{(-1)^{\nu+1}}\prod_{p\in I_{\nu,2}\cup I_{\nu,4}}\left(1-\frac{1}{p}\right)^{(-1)^{\nu}}= 1+O\left(\frac{1}{x_i}\right)\: .
\]
In addition, the lengths of $I_{i, 1}$ and $I_{i, 4}$ are $O(\frac{ix_{i}}{\log^{1/3}x_{i}})$ and each of them contains $O(\frac{ix_{i}}{\log^{4/3}x_{i}})$ (ordinary rational) primes.
\end{lemma}

We define $x_k^-$ to be the least integer in $I_{k, 1}$, and $x_k^+$
the largest integer in $I_{k, 4}$. It follows that
$\frac{x_k}{\log^{1/3}x_k}\leq x_k^+-x_k^-=O(\frac{kx_k}{\log^{1/3}x_k})$. Since $k<\log^{1/6}x_{k}$, we therefore have that $x_{k}^{+}<2x_{k}$ and $x_{k}^{-}>2^{-1}x_{k}$, for sufficiently large $k$. We may thus assume that these properties hold for all $k$.

The sequence of generalized primes $P^{\ast}=\left\{p_{\nu}\right\}_{\nu=1}^{\infty}$ is then constructed as follows. We use the term `prime number' for the ordinary rational primes and `prime element' for the elements of $P^{\ast}$. Take one prime
element 
$p$ for each prime number $p$ which is not in any of the intervals
$I_{i, j}\: $. If $i$ is even, take no prime elements in $I_{i, 2}\cup I_{i, 4}$ and two prime elements
$p$ for all prime numbers $p$ which are in one of the intervals $I_{i,
  1}, I_{i, 3}$. If $i$ is odd, no prime elements in $I_{i, 1}\cup I_{i, 3}$ and two prime elements
for all prime numbers $p$ which belong to one of the intervals $I_{i,
  2}, I_{i, 4}$. As previously mentioned, we simplify the notation and write $\pi(x)=\pi_{P^{\ast}}(x)$ and $N(x)=N_{P^{\ast}}(x)$ for the counting functions of $P^{\ast}$ and its associated generalized integer counting function. The rest of the section is dedicated to proving that $N$ and $\pi$ have the properties stated in Proposition \ref{propExample}. We actually show something stronger:

\begin{proposition}
\label{Example}
We have $N(x)=x+\Omega(x/\log^{4/3}x)$; however for an arbitrary $\varepsilon > 0$,
$$N(x)=x+O\left(\frac{x}{\log^{5/3-\varepsilon} x}\right) \ \ \  (\mathrm{C},1)\: , $$
i.e., its first order Ces\`{a}ro-mean $\overline{N}$ has asymptotics
\begin{equation}
\label{bpexeq1}
\overline{N}(x):=\int_{1}^{x}\frac{N(t)}{t}\:\mathrm{d}t=x+O\left(\frac{x}{\log^{5/3- \varepsilon} x}\right)\: .
\end{equation}
For this system, 
$$
\pi(x)=\frac{x}{\log x}+O\left(\frac{x\log\log x}{\log^{4/3} x}\right)\: .
$$
\end{proposition}
The asymptotic bound for the prime counting function $\pi$ of our generalized prime set $P^{\ast}$
follows immediately from the definition of $P^{\ast}$ and the
classical prime number theorem. The non-trivial part in the proof of Proposition \ref{Example} is to establish the asymptotic formulas for $N$ and $\overline{N}$. 

To achieve further progress, we introduce a family of generalized prime number systems approximating $P^{\ast}$. We define the generalized prime set $P^{\ast}_{k}$ by means of the same construction used for $P^{\ast}$, but only  taking the intervals $I_{i, j}$
with $i\leq k$ into account; furthermore, we write $N_k(x)=N_{P_{k}^{\ast}}(x)$.

We first try to control the growth $N_k(x)$ on suitable large intervals. For this we will use a result from the theory of integers without large prime factors \cite{hildebrand-tenenbaum}. This theory studies the function 
\[ \Psi(x,y) = \#\{1 \leq n \leq x: P(n) \leq y\} \:,
\]
where $P(n)$ denotes the largest prime factor of $n$ with the convention $P(1)=1$. This function is well studied  \cite{hildebrand-tenenbaum} and we will only use  the simple estimate \cite[Eqn. (1.4)]{hildebrand-tenenbaum}:
\begin{equation}
\label{eqestimatepsi}
\Psi(x,y) \ll xe^{-\log x/2\log y}\log y\: .
\end{equation}

A weaker version of the following lemma was stated in \cite{s-v}, but the proof given there contains a mistake. Furthermore, the range of validity for the estimates in \cite[Lemma 6.3]{s-v} appears to be too weak to lead to a proof of the Ces\`{a}ro estimate (\ref{bpexeq1}). We correct the error in the proof and show the assertions in a broader range.
\begin{lemma}
\label{lemmaexample} Let $\eta > 1$.
If $\exp(\log^{\eta} x_{k})\leq x<\exp(x_k^{3/5})$, then we have
\begin{equation}
\label{bpexeq2}
N_k(x)=x+O\left(\frac{x}{\log^{5/3} x}\right)
\end{equation}
and 
\begin{equation}
\label{bpexeq3}
\overline{N}_{k}(x):=\int_{1}^{x}\frac{N_{k}(t)}{t}\:\mathrm{d}t= x+O\left(\frac{x}{\log^{5/3} x}\right)\: ,
\end{equation}
for all sufficiently large $k$.
\end{lemma}
\begin{proof}
Let $f(n)$ be the number of representations of $n$ as finite products of elements of $P^{\ast}_{k}$. Note that $N_{k}(x)=\sum_{n\leq x}f(n)$. Setting $f(1)=1$,
the function $f(n)$ becomes multiplicative and we have
\[
f(p^\alpha) = \begin{cases} \alpha+1\ , & \mbox{if } \exists 2i\leq k: p\in
  I_{2i,1}\cup I_{2i,3}\: ,\\
0\ , & \mbox{if }\exists 2i\leq k: p\in I_{2i,2}\cup I_{2i,4}\: ,\\
0\ , & \mbox{if }\exists 2i+1\leq k: p\in I_{2i+1,1}\cup I_{2i+1,3}\: ,\\
\alpha+1\ , & \mbox{if } \exists 2i+1\leq k: p\in
  I_{2i+1,2}\cup I_{2i+1,4}\: ,\\

1\ , & \mbox{otherwise}.
\end{cases}
\]
We also introduce the multiplicative function $g(n)=\sum_{d|n}\mu(n/d)f(d)$. The values of $g$ at powers of prime numbers are easily seen to be
\[
g(p^\alpha) = 
\begin{cases}
1\ , &\mbox{if } f(p)=2\: ,\\
-1\ ,& \mbox{if } f(p)=0\mbox{ and } \alpha=1\: ,\\
0\ , & \mbox{otherwise}.
\end{cases}
\]

Denote by $\mathcal{H}_{k}$ the set of all integers which have
only prime divisors in $\bigcup_{i\leq k} I_{i, j}$, and for each
integer $n$, let $n_{\mathcal{H}_{k}}$ be the largest divisor of $n$
belonging to $\mathcal{H}_{k}$. We have 
\begin{align*}
N_k(x) &= \sum_{m\in\mathcal{H}_{k}} \underset{n_{\mathcal{H}_{k}}=m}
{\sum_{n\leq x}} f(m)
=
\sum_{m\in\mathcal{H}_{k}} \underset{m|n}
{\sum_{n\leq x}} g(m)
\\
&=
\sum_{m\in\mathcal{H}_{k}} g(m) \left[\frac{x}{m}\right]
 \\
 &
 =
  x\sum_{m\in\mathcal{H}_{k}}\frac{g(m)}{m} - x \underset{d > x} {\sum_{m\in\mathcal{H}_{k}}}\frac{g(m)}{m} + 
O\left(|\mathcal{H}_{k}\cap[1, x]|\right),
\end{align*}
and, since 
\begin{equation*}
\sum_{m\in\mathcal{H}_{k}}\frac{g(m)}{m} = \prod_{i=1}^{k} \prod_{p\in I_{i,1}\cup I_{i,3}}\left(1-\frac{1}{p}\right)^{(-1)^{i+1}}\prod_{p\in I_{i,2}\cup I_{i,4}}\left(1-\frac{1}{p}\right)^{(-1)^{i}}\: ,
\end{equation*}
we thus obtain
\begin{equation*}
N_k(x) 
  = x + O\left(\frac{x}{x_k}\right)+ O\left(|\mathcal{H}_{k}\cap[1, x]|\right) - x \underset{d > x} {\sum_{m\in\mathcal{H}_{k}}}\frac{g(m)}{m} \: .
\end{equation*}
The first error term is negligible because $x<\exp(x_k^{3/5})$. For the estimation of the remaining two terms we use the function $\Psi$. Any element of $\mathcal{H}_{k}$ has only prime divisors below $2x_{k}$. Using this observation and employing the estimate (\ref{eqestimatepsi}), we find that
\begin{align*}
\left|\mathcal{H}_{k}\cap [1,x]\right|& \ll x^{1-\frac{1}{2\log (2x_{k})}}\log x_{k}
\\
&\ll \frac{x}{\log^{5/3}x} \left(\frac{(x_{k}\log x_{k})^{2\log (2x_{k})}}{x} \right)^{\frac{1}{2\log (2x_{k})}}
\\
&
\ll
 \frac{x}{\log^{5/3}x}\: ,  \ \ \ \mbox{for } \exp(8\log^{2} x_{k})\leq x<\exp(x_k^{3/5})\: .
\end{align*}
Similarly, we can extend the bound to the broader region,
\begin{align*}
\left|\mathcal{H}_{k}\cap [1,x]\right|& \ll x^{1-\frac{1}{2\log (2x_{k})}}\log x_{k}
\\
&\ll \frac{x}{\log^{5/3}x} \left(\frac{(\log^{13/3} x_{k})^{2\log (2x_{k})}}{x} \right)^{\frac{1}{2\log (2x_{k})}}
\\
&
\ll
 \frac{x}{\log^{5/3}x}\: ,  \ \ \ \mbox{for } \exp(\log^{\eta} x_{k})\leq x<\exp(8 \log^{2} x_k)\: ,
\end{align*}
which is valid for all sufficiently large $k$. For the other term,
\begin{align*}
	\left| \sum_{\substack{d \in \Hk \\ d > x }} \frac{g(d)}{d}\right| & \leq  \sum_{\substack{P(d) \leq 2x_{k}  \\ d > x }} \frac{1}{d}= \int^{\infty}_{x^{-}} \frac{1}{t}\: \mathrm{d}\Psi(t,2x_{k})\\
	& = \lim_{t \rightarrow \infty}\frac{\Psi(t,2x_{k})}{t} - \frac{\Psi(x,2x_{k})}{x} + \int^{\infty}_{x} \frac{\Psi(t,2x_{k})}{t^{2}} \:\mathrm{d}t\:.
\end{align*}
The limit term equals $0$ because it is $O(t^{-1/2\log( 2x_{k})} \log 2x_{k})$ by (\ref{eqestimatepsi}). The second term is negligible because it is a negative term in a positive result. It remains to bound the integral:
\begin{align*}
	\int^{\infty}_{x} \frac{\Psi(t,2x_{k})}{t^{2}}\: \mathrm{d}t & \ll\int^{\infty}_{x} t^{-1-\frac{1}{2\log (2x_{k})}} \log 2x_{k} \mathrm{d}t\\
	& \ll x^{-\frac{1}{2\log(2x_{k})}}\log^{2} x_{k}\\
	& \ll\frac{1}{\log^{5/3} x} \quad \text{because }\exp( \log^{\eta} x_{k}) \leq x \leq \exp(x_{k}^{3/5})\: ,
\end{align*}
where the last inequality is deduced in the same way as above. This concludes the proof of (\ref{bpexeq2}).
We now address the Ces\`{a}ro estimate. Using the estimates already found for $N_{k}(x)$, we find 
\[
\overline{N}_{k}(x)-x
=O\left(\frac{x}{\log^{5/3}x}\right)+O\left(\int_{1}^{x}\frac{\left|\mathcal{H}_{k}\cap[1,t]\right|}{t}\:\mathrm{d}t\right) +  O\left(  \int^{x}_{1}\int^{\infty}_{t} \frac{\Psi(s,2x_{k})}{s^{2}}\: \mathrm{d}s\mathrm{d}t \right)\: .
\]
We bound the double integral in the given range. The other term can be treated similarly. We obtain
\begin{align*}
	\int^{x}_{1}\int^{\infty}_{t} \frac{\Psi(s,2x_{k})}{s^{2}}\: \mathrm{d}s\mathrm{d}t  & \ll \int^{x}_{1} \int^{\infty}_{t} s^{-1-\frac{1}{2 \log (2x_{k})}} \log (2x_{k})\: \mathrm{d}s\mathrm{d}t\\
	& = \int^{x}_{1} t^{-\frac{1}{2 \log (2x_{k})}} 2\log^{2} (2x_{k})\mathrm{d}t\\
	& \ll x^{1-\frac{1}{2 \log (2x_{k})}}\log^{3}(2x_{k})= O\left(\frac{x}{\log^{5/3}x}\right)\: ,
\end{align*}
where again the last step is shown by considering the regions $\exp(8\log^{2} x_{k})\leq x<\exp(x_k^{3/5})$ and $\exp(\log^{\eta} x_{k})\leq x<\exp(8 \log^{2} x_k)$ separately.
\end{proof}
 
We end the article with the proof of Proposition \ref{Example}.
\begin{proof}[Proof of Proposition \ref{Example}] We choose $\eta$ smaller than $\frac{5/3}{5/3 - \varepsilon}$ in Lemma \ref{lemmaexample}. The $\Omega$-estimate for $N(x)$ follows almost immediately from (\ref{bpexeq2}). For
$x<x_{k+1}^{+}$, we have $N(x)=N_k(x)$ with the exception of the missing
and doubled primes from $[x_{k+1}^{-},x_{k+1}^{+}]$. Observe that, because $x_{k+1}=\lfloor\exp(x_{k}^{1/4}\log 2)\rfloor$,
$$[x_{k+1}^{-},x_{k+1}^{+}]\subset[\exp\left(\log^{\eta} x_k\right),\exp(x_k^{3/5}))\: .$$
Since we changed more than $\frac{x}{4\log^{4/3} x}$ primes when $x$ is  the upper bound of either the interval $I_{k+1,1}$ or $I_{k+1,2}$,
we obtain from Lemma \ref{lemmaexample} that $|N(x)-x|$ becomes as large as 
$\frac{x}{8\log^{4/3 }x}$ infinitely often as $x\to\infty$.

It remains to show (\ref{bpexeq1}). We bound the Ces\`{a}ro means of $N$ in the range $x_k^-\leq
x<x_{k+1}^{-}.$
We start by observing that $N(x)=N_{k}(x)$ within this range, so (\ref{bpexeq3}) gives (\ref{bpexeq1}) for $\exp(\log^{\eta} x_{k})\leq x<x_{k+1}^{-}$. Assume now that
$$x_k^-\leq x<\exp(\log^{\eta} x_{k})\: .$$ 
Lemma \ref{lemmaexample} implies that
$$
\overline{N}_{k-1}(x)=\int_{1}^{x}\frac{N_{k-1}(t)}{t}\:\mathrm{d}t=x+O\left(\frac{x}{\log^{5/3}x}\right)\: ,
$$
because, by construction of the sequence, the interval $[x_{k}^{-},\exp(\log^{\eta} x_{k})]$ is contained in 
$[\exp(\log^{\eta} x_{k-1}),\exp(x_{k-1}^{3/5})]$. Therefore, it suffices to prove that
\begin{equation}
\label{bpexeq4}
\overline {N}_{k}(x)-\overline{N}_{k-1}(x)=\int_{x_{k}^{-}}^{x}\frac{N_{k}(t)-N_{k-1}(t)}{t}\:\mathrm{d}t
\end{equation}
has growth order $O(\frac{x}{\log^{5/3 - \varepsilon}x})$ in the interval $[x_{k}^{-},\exp(\log^{\eta} x_{k})]$. Note that only the intervals $\nu\cdot (I_{k,1}
\cup\ldots \cup I_{k, 4})$ contribute to the integral (\ref{bpexeq4}) with $\nu$ a generalized integer from the number system generated by $P^{\ast}_{k}$. Only the generalized integers $ \nu \leq x/x_{k}^{-}$ deliver a contribution. There are at most $O(x/x_{k}^{-}) = O(x/x_{k})$ such integers. The contribution of one such a generalized integer is then
\begin{align*}
	O\left(\frac{kx_{k}}{\log^{4/3}x_{k}}\right) \int^{\nu x_{k}^{+}}_{\nu x_{k}^{-}} \frac{\mathrm{d}t}{t} &= O\left(\frac{k^{2}x_{k}}{\log^{5/3}x_{k}}\right),
\end{align*}
where we have used the fact that the length of the intervals $I_{k,i}$ are $O(kx_{k}\log^{-1/3}x_{k})$ as derived in Lemma \ref{bpexampleclaim}. In total the integral is bounded by
\[ O\left(\frac{x}{x_{k}}\right)O\left(\frac{k^{2}x_{k}}{\log^{5/3}x_{k}}\right) = O\left(\frac{k^{2} x}{\log^{5/(3\eta)}x}\right) = O\left( \frac{x}{\log^{5/3 - \varepsilon}x}\right)\: .
\]
\end{proof}

\end{document}